\documentclass[reqno]{amsart}
\usepackage{amsfonts, mathrsfs}
\usepackage{amsthm}
\usepackage{amsmath}
\usepackage{graphicx}
\usepackage{amscd,amssymb}

\begin{document}
\title[On The Extended Incomplete Pochhammer ....]
{On The Extended Incomplete Pochhammer Symbols and Hypergeometric Functions}

\author[]{Rakesh K. Parmar$^\dag$}
\address{$^\dag$ Department of Mathematics, Government College of Engineering and Technology, Bikaner-334004,
Rajasthan, India} \email{rakeshparmar27@gmail.com}

\author[R.K. Parmar, R.K. Raina]{R.K. Raina$^{\ddag, \P}$}
\address{$^\ddag$ M.P. University of Agriculture and Technology, Udaipur-313001, Rajasthan, India}
\address{$^\P$ \emph{Present address:} 10/11, Ganpati Vihar, Opposite Sector 5\\
Udaipur-313002, Rajasthan, India.} \email{rkraina\_7@hotmail.com}
\subjclass[2010]{ Primary 33B20, 33C05, 33C20; Secondary 33B99, 33C99.}
\keywords{Generalized incomplete gamma functions; Generalized Incomplete Pochhammer symbols;  Incomplete hypergeometric functions; Laguerre polynomials; Bessel functions; Generating function}

\begin{abstract}
In this paper, we first introduce certain forms of extended incomplete Pochhammer symbols which are then used to define families of extended incomplete generalized hypergeometric functions. For these functions, we investigate various properties including the integral representations, derivative formula, certain generating function and fractional integrals (and derivatives) relationships. Some special cases of the main results are also deduced.
\end{abstract}

\maketitle
\numberwithin{equation}{section}
\newtheorem{theorem}{Theorem}
\newtheorem{lemma}{Lemma}
\newtheorem{proposition}{Proposition}
\newtheorem{corollary}{Corollary}[theorem]
\newtheorem{remark}{Remark}
\allowdisplaybreaks

\section{ Introduction and Preliminaries}
Extensions, generalizations and unifications of Euler's Gamma function together with the set of related higher transcendental special functions were studied recently by Chaudhry and Zubair in  \cite{Ch-Zu}. In particular, Chaudhry and Zubair \cite[p. 100, Eqns. (4-5)]{Ch-Zu-94} presented the $p$--extension of the \emph{familiar}
incomplete Gamma functions $\gamma (\nu,x)$ and $\Gamma (\nu,x)$ by
\begin{equation}\label{Incomplete-gamma-p}
\gamma(\nu,x;p):=\int_0^x\, t^{\nu-1}\, {\rm e}^{-t- \frac{p}{t}} \;dt\qquad(\Re(p)>0;\, p=0,\,\, \Re(\nu)>0)
\end{equation}
and
\begin{equation}\label{Incomplete-Gamma-p}
\Gamma(\nu,x;p):=\int_x^\infty\, t^{\nu-1}\,{\rm e}^{-t- \frac{p}{t}}\;dt\qquad(\Re(p)>0),
\end{equation}
respectively, satisfying the following decomposition formula:
\begin{equation}\label{gamma+Gamma-function-p}
\gamma (s,x;p) + \Gamma (s,x;p)\equiv \Gamma_p (s)=\int_0^\infty\, t^{s-1}\, {\rm e}^{-t- \frac{p}{t}}\;dt=2p^{s/2}K_{s}\left(2\sqrt{p}\right) \qquad (\Re (p)>0),
\end{equation}
where $K_\nu(x)$ denotes the familiar modified Bessel function \cite{Ol-Lo-Bo}.

In this paper, we first introduce the family of the extended incomplete Pochhammer symbols  $ \left(\lambda; x ,p \right)_{\nu} $ and $ \left[\lambda; x ,p \right]_{\nu}$ by means of generalized incomplete Gamma functions $\gamma (s,x;p)$ and $\Gamma (s,x;p)$, respectively. We then derive its useful properties and make use of
it to define and investigate the family of the extended incomplete hypergeometric functions $\;_r\gamma_s^{p}(z)$ and $\;_r\Gamma_s^{p}(z)$ with $r$ numerator and
$s$ denominator parameters. For these extended incomplete hypergeometric functions, we derive various integral representations involving higher transcendental functions. We also derive a derivative formula, certain generating function relationships and fractional integrals (and derivatives) involving the extended incomplete hypergeometric functions. For other investigations on the subject and related areas, one may also refer to the works in \cite{Cent,Sr-Ch-Ag,Sr-Ce-Ki,RS2013-RJMP,SCho}.
\section{ The Extended Incomplete Pochhammer Symbols}
In terms of the generalized incomplete Gamma functions $\gamma (s,x;p)$ and $\Gamma (s,x;p)$ defined by \eqref{Incomplete-gamma-p} and \eqref{Incomplete-Gamma-p}, respectively, the new forms of the Pochhammer symbols
$ \left(\lambda; x ,p \right)_{\nu} $ and $ \left[\lambda; x ,p \right]_{\nu}$
$\;(\lambda; \,  \nu \in \mathbb{C};\;p\geqq  0)$
may be defined by
\begin{equation}\label{gips-p-1}
\left(\lambda;x,p\right)_{\nu} := \frac{\gamma (\lambda +\nu,x;p) }{\Gamma (\lambda)}\qquad(\lambda, \nu \in \mathbb{C}; \,\, p\geqq  0 )
\end{equation}
and
\begin{equation}\label{gips-p-2}
\left[\lambda;x,p\right]_{\nu} := \frac{\Gamma (\lambda +\nu,x;p) }{\Gamma (\lambda)}\qquad(\lambda, \nu \in \mathbb{C}; \,\,p\geqq  0 ).
\end{equation}
Obviously, then these new forms of the Pochhammer symbols $\left(\lambda;x,p \right)_{\nu}$ and $\left[\lambda;x,p \right]_{\nu}$ then satisfy the following decomposition relation:
\begin{equation}\label{gips-p-1+gips-p-2-symbol-p}
\left(\lambda;x;p \right)_{\nu} + \left[\lambda;x ;p \right]_{\nu}\equiv\left(\lambda;p\right)_{\nu}:=\left\{
\begin{array}{ll}
\displaystyle \frac{\Gamma_{p} (\lambda + \nu)}{\Gamma (\lambda )} & \qquad (\Re(p) > 0; \,\, \lambda, \,\, \nu  \in \mathbb{C})\\
\\
\int_0^\infty\, t^{\lambda+\nu-1}\, e^{-t-\frac{p}{t}}\,dt & \qquad (\Re(p) > 0).
\end{array}
\right.,
\end{equation}
where $(\lambda;p)_{\nu}$ is the generalized Pochhammer symbol (\cite[p. 485, Eqn. (8)]{Sr-Ce-Ki}).

In view of the relations \eqref{Incomplete-gamma-p}, \eqref{gips-p-1} and \eqref{Incomplete-Gamma-p}, \eqref{gips-p-2}, we can at once
have the integral representations of the new forms of the Pochhammer symbols contained in the following lemma.
\begin{lemma} For $\Re(p)>0$, we have the following integral representations
for $ \left(\lambda; x ,p \right)_{\nu} $ and $ \left[\lambda; x ,p \right]_{\nu}$:
\begin{equation}\label{gips-p-I-1}
\left(\lambda;x,p\right)_{\nu} := \frac{1}{\Gamma (\lambda)}\int_0^x\, t^{\lambda+\nu-1}\,e^{-t-\frac{p}{t}}\,dt
\end{equation}
and
\begin{equation}\label{gips-p-I-2}
\left[\lambda;x,p\right]_{\nu} := \frac{1}{\Gamma (\lambda)}\int_x^\infty\, t^{\lambda+\nu-1}\,e^{-t-\frac{p}{t}}\,dt.
\end{equation}
\end{lemma}
Also, by writing the relations \eqref{gips-p-1} and \eqref{gips-p-2} as
$$\left(\lambda; x ,p \right)_{\nu}  = \frac{\Gamma(\lambda+\nu)}{\Gamma (\lambda)}\frac{ \gamma(x,p,\nu)}{\Gamma(\lambda+\nu)}
\qquad \text{and}\qquad
\left[\lambda; x ,p \right]_{\nu}  = \frac{\Gamma(\lambda+\nu)}{\Gamma (\lambda)} \frac{\Gamma[x,p,\nu]}{\Gamma(\lambda+\nu)},$$
we have the following lemma for the
extended incomplete Pochhammer symbols \eqref{gips-p-1} and \eqref{gips-p-2}.
\begin{lemma} Let $\lambda \in \mathbb{C}; m, n \in \mathbb{N}_{0}$ and $\Re(p)>0$, then for $ \left(\lambda; x ,p \right)_{\nu} $ and $ \left[\lambda; x ,p \right]_{\nu}$, we have
\begin{equation}\label{Pochhammer-relation}  
(\lambda ; x,p)_{n+m} = (\lambda)_{n}\; (\lambda +n ;x,p)_{m}
\end{equation}
and
\begin{equation}\label{Pochhammer-relation-2}  
[\lambda ; x,p]_{n+m} = (\lambda)_{n}\; [\lambda +n ;x,p]_{m}.
\end{equation}
\end{lemma}
\section{ The Extended Incomplete Hypergeometric Functions}
In  terms of the extended incomplete forms of the Pochhammer symbols $(\lambda;x,p)_{\nu}$ and $[\lambda;x,p]_{\nu}$ defined, respectively, by \eqref{gips-p-1} and \eqref{gips-p-2}, we introduce two families of the extended incomplete generalized hypergeometric functions $_{r}\gamma_{s}^{p}(z)$ and $_{r}\Gamma_{s}^{p}(z)$, respectively, involving $r$-numerator and $s$-denominator parameters as follows:
\begin{equation}\label{eighf-1}
\;_{r}\gamma_{s}^{p}(z)=\,_{r}\gamma _{s} \left[\begin{array}{rr}
(\alpha_{1},x;p),\alpha_{2},\cdots,\alpha_{r};          \\
\beta_{1},\cdots,\beta_{s};    \\
\end{array} z \right] = \sum_{n=0}^{\infty} \frac{(\alpha_{1};x,p)_{n}(\alpha_{2})_{n}\cdots(\alpha_{r})_{n}}{(\beta_{1})_{n}\cdots(\beta_{s})_{n}} \;\frac{z^{n}}{n!}
\end{equation}
and
\begin{equation}\label{eighf-2}
\;_{r}\Gamma_{s}^{p}(z)=\,_{r}\Gamma _{s} \left[\begin{array}{rr}
(\alpha_{1},x;p),\alpha_{2},\cdots,\alpha_{r};          \\
\beta_{1},\cdots,\beta_{s};    \\
\end{array} z \right] = \sum_{n=0}^{\infty} \frac{[\alpha_{1};x,p]_{n}(\alpha_{2})_{n}\cdots(\alpha_{r})_{n}}{(\beta_{1})_{n}\cdots(\beta_{s})_{n}}\; \frac{z^{n}}{n!},
\end{equation}
where $\alpha_{1},\ldots,\,\alpha_{r} \in \mathbb{C}$
and $ \beta_{1},\ldots, \beta_{s}\in \mathbb{C} \setminus \mathbb{Z}^{-}_{0}$ provided that the series on the right-hand side of \eqref{eighf-1} and \eqref{eighf-2} converge.

Following \eqref{gips-p-1+gips-p-2-symbol-p}, these families of extended incomplete generalized hypergeometric functions satisfy the following decomposition formula:
\begin{align}\label{decom-pFq}
&{}_{r}\gamma _{s} \left[\begin{array}{rr}
(\alpha_{1},x;p),\alpha_{2},\ldots,\alpha_{r};          \\
\beta_{1},\ldots,\beta_{s};    \\
\end{array} z \right]
+{}_{r}\Gamma _{s} \left[\begin{array}{rr}
(\alpha_{1},x;p),\alpha_{2},\ldots,\alpha_{r};          \\
\beta_{1},\cdots,\beta_{s};    \\
\end{array} z \right]\notag \\
& \qquad \qquad= \sum_{n=0}^{\infty} \frac{(\alpha_{1};p)_{n}(\alpha_{2})_{n}\cdots(\alpha_{r})_{n}}{(\beta_{1})_{n}\cdots(\beta_{s})_{n}} \;\frac{z^{n}}{n!}  ={}_rF_{s} \left[\begin{array}{r}
(\alpha_{1},p),\alpha_{2},\ldots,\alpha_{r};          \\
\beta_{1},\ldots,\beta_{s};    \\
\end{array} z \right],
\end{align}
which is an extension of the generalized hypergeometric function $_{r}F_{s}(z)$ studied in (\cite[p. 487, Eq. (15)]{Sr-Ce-Ki}). In view of the decomposition formula \eqref{decom-pFq}, it is sufficient to discuss the properties and characteristics of the extended incomplete generalized hypergeometric functions $\;_{r}\Gamma_{s}^{p}(z)$. The corresponding extensions of the incomplete Gaussian hypergeometric function $_{2}\Gamma _{1}^{p}(z)$ and the confluent (Kummer's) hypergeometric function $_{1}\Gamma _{1}^{p}(z)$ can, respectively,  be expressed by
\begin{equation}\label{incomplete-Gauss-2f1}
_{2}\Gamma _{1}^{p}(z)=\;_{2}\Gamma _{1} \left[\begin{array}{rr}
(\alpha,x;p),\beta;          \\
\gamma;    \\
\end{array} z \right] = \sum_{n=0}^{\infty} \frac{[\alpha;x,p]_{n}(\beta)_{n}}{(\gamma)_{n}}\; \frac{z^{n}}{n!}
\end{equation}
and
\begin{equation}\label{incomplete-confluent-1f1}
_{1}\Gamma _{1}^{p}(z)=\;_{1}\Gamma _{1} \left[\begin{array}{rr}
(\alpha,x;p);          \\
\gamma;    \\
\end{array} z \right] = \sum_{n=0}^{\infty} \frac{[\alpha;x,p]_{n}}{(\gamma)_{n}}\; \frac{z^{n}}{n!}.
\end{equation}
\begin{theorem}\label{Theorem-1} The following integral representation for ${}_{r}\Gamma _{s}^{p}(z)$ defined by \eqref{eighf-2}
 holds true:
\begin{align}\label{integral-ighf-p}
&\;_{r}\Gamma _{s} \left[\begin{array}{rr}
(\alpha_{1},x;p),\alpha_{2},\cdots,\alpha_{r};          \\
\beta_{1},\cdots,\beta_{s};    \\
\end{array} z \right] = \frac{1}{\Gamma (a_{1})} \int_{x}^{\infty} t^{a_{1}-1}\;e^{-t-\frac{p}{t}} \; _{r-1}F_{s} \left[\begin{array}{rr}
\alpha_{2},\cdots,\alpha_{r};          \\
\beta_{1},\cdots,\beta_{s};    \\
\end{array} zt \right] dt
\end{align}
$$\big(\Re(p)>0;\;\;\Re (a_{1}) > 0 \;\text{when}\;p=0\;\text{and}\; x=0 \big).$$
\end{theorem}
\begin{proof}
Using the integral representation of the extended incomplete Pochhammer symbol $[a_{1};x,p]_{n}$ defined by \eqref{gips-p-I-2} and then changing the order of summation and integration, we obtain
\begin{align}\label{integral-ilf-2}
&\;_{r}\Gamma _{s} \left[\begin{array}{rr}
(\alpha_{1},x;p),\alpha_{2},\cdots,\alpha_{r};          \\
\beta_{1},\cdots,\beta_{s};    \\
\end{array} z \right] =\frac{1}{\Gamma(\alpha_{1})}\;\int_{x}^{\infty} t^{\alpha_{1}-1}e^{-t-\frac{p}{t}}\;
\sum_{n=0}^{\infty}\frac{(\alpha_{2})_{n}\cdots(\alpha_{r})_{n}}{(\beta_{1})_{n}\cdots(\beta_{s})_{n}} \frac{(zt)^{n}}{n!} dt,
\end{align}
which is precisely the second member of the assertion \eqref{integral-ighf-p}. This establishes Theorem \ref{Theorem-1}.
\end{proof}
\begin{theorem}\label{Theorem-2} The following integral representation for ${}_{r}\Gamma _{s}^{p}(z)$ defined by \eqref{eighf-2}
holds true:
\begin{align}\label{integral-ilf-3}
&\hspace{-20mm}\;_{r}\Gamma _{s} \left[\begin{array}{rr}
(\alpha_{1},x;p),\alpha_{2},\cdots,\alpha_{r-1},\beta;          \\
\beta_{1},\cdots,\beta_{s-1},\gamma;    \\
\end{array} z \right]= \frac{1}{B(\beta,\gamma-\beta)} \int_{0}^{1} t^{\beta-1} (1-t)^{\gamma-\beta-1}\;   \notag\\
&\hskip 35mm \times\;_{r-1}\Gamma_{s-1} \left[\begin{array}{rr}
(\alpha_{1},x;p),\alpha_{2},\cdots,\alpha_{r-1};          \\
\beta_{1},\cdots,\beta_{s-1};    \\
\end{array} zt \right] dt
\end{align}
$$\big(\Re(p)>0;\;\;\Re (\gamma)>\Re (\beta) > 0 \;\text{when}\; p=0 \,\,\text{and}\,\, x=0\big).$$
\end{theorem}
\begin{proof}
By considering the following elementary integral form of the Beta function $B(\alpha,\beta)$:
$$\frac{(\beta)_{n}}{(\gamma)_{n}}=\frac{B (\beta+n, \gamma - \beta)}{B (\beta, \gamma - \beta)}=\frac{1}{B (\beta, \gamma - \beta)} \int_{0}^{1} t^{\beta+n-1} (1 - t)^{\gamma - \beta - 1}\; dt ,$$
$$\big(\Re(\gamma)>\Re(\beta)>0\big)$$
in the left-hand side of \eqref{integral-ilf-3} and using \eqref{eighf-2}, we get the desired integral representation
\eqref{integral-ilf-3}.
\end{proof}
The following corollaries are easy consequences of the results  \eqref{incomplete-Gauss-2f1} and \eqref{incomplete-confluent-1f1}.
\begin{corollary} Each of the following integral representation for ${}_{2}\Gamma _{1}^{p}(z)$ and ${}_{1}\Gamma _{1}^{p}(z)$ in \eqref{incomplete-Gauss-2f1} and \eqref{incomplete-confluent-1f1}
 holds true:
\begin{align}\label{integral-ghf-p}
&\;_{2}\Gamma _{1} \left[\begin{array}{rr}
(\alpha,x;p),\beta;          \\
\gamma;    \\
\end{array} z \right] = \frac{1}{\Gamma (\alpha)} \int_{x}^{\infty}  t^{\alpha-1}\;e^{-t-\frac{p}{t}} \; _{1}F_{1}\left[ \begin{array}{r}
\beta;          \\
\gamma;          \\
\end{array} zt \right]dt,
\end{align}
\begin{align}\label{integral-chf-p}
&\;_{1}\Gamma _{1} \left[\begin{array}{rr}
(\alpha,x;p);          \\
\gamma;    \\
\end{array} z \right] = \frac{1}{\Gamma (\alpha)} \int_{x}^{\infty}  t^{\alpha-1}\;e^{-t-\frac{p}{t}} \; _{0}F_{1}\left[ \begin{array}{r}
\overline{\hspace{5mm}};          \\
\gamma;          \\
\end{array} zt \right]dt
\end{align}
and
\begin{align}\label{integral-ghf-p-1}
&\;_{2}\Gamma _{1} \left[\begin{array}{rr}
(\alpha,x;p),\beta;          \\
\gamma;    \\
\end{array} z \right] =\frac{1}{B(\beta,\gamma-\beta)} \int_{0}^{1} t^{\beta-1} (1-t)^{\gamma-\beta-1}\; _{1}\Gamma_{0} \left[\begin{array}{rr}
(\alpha,x;p);          \\
\overline{\hspace{5mm}};    \\
\end{array} zt \right] dt.
\end{align}
\end{corollary}
The well known Laguerre polynomial $L_{n}^{(\alpha)}(x)$ of order (index) $\alpha$ and degree $n$ in $x$, the incomplete Gamma function $\gamma (k,x)$ and the Bessel function $J_{\nu} (z) $ are, respectively, expressible in terms of the hypergeometric functions by
(see, \emph{e.g.}, ~\cite{Ol-Lo-Bo})
\begin{equation}\label{Laguerre-pFq}
L_{n}^{(\alpha)}(x)=\frac{(\alpha+1)_{n}}{n!}\;_{1}F_{1} (-n; \alpha +1;x),
\end{equation}
\begin{equation}\label{gamma-pFq}
\;_{1}F_{1} (\kappa; \kappa+1;-x) = \kappa x^{-\kappa}\; \gamma (\kappa,x)
\end{equation}
and
\begin{equation}\label{J-pFq}
J_{\nu} (z) = \frac{(\frac{z}{2})^{\nu}}{\Gamma (\nu + 1)} \; _{0}F_{1} \left(\overline{\hspace{5mm}}\;; \nu+1; -\frac{1}{4}z^{2} \right) \qquad (\nu\in \mathbb{C}\setminus \mathbb{Z}^{-}).
\end{equation}
Now, applying the relationships \eqref{Laguerre-pFq} and \eqref{gamma-pFq}  to \eqref{integral-ghf-p}
and \eqref{J-pFq} to  \eqref{integral-chf-p},
we obtain integral representations for the extended incomplete hypergeometric functions defined by
\eqref{incomplete-Gauss-2f1} and \eqref{incomplete-confluent-1f1}, which are asserted by the Corollary \ref{cor-1} below.
\begin{corollary}\label{cor-1} Each of the following integral representations hold true:
\begin{align}\label{Lau-integral}  
&\;_{2}\Gamma _{1} \left[\begin{array}{rr}
(\alpha,x;p),-m;          \\
\gamma+1;    \\
\end{array} z \right]= \frac{m!}{ (\gamma+1)_{m}\Gamma (\alpha)} \int_{x}^{\infty} t^{\alpha-1} \;e^{-t-\frac{p}{t}} \;\; L_{m}^{(\gamma)}(zt)dt
\end{align}
\begin{align}\label{gamma-integral}
&\;_{2}\Gamma _{1} \left[\begin{array}{rr}
(\alpha,x;p),\beta;          \\
\beta+1;    \\
\end{array} -z \right] = \frac{\beta\;z^{-\beta}}{\Gamma (\alpha)} \int_{x}^{\infty}t^{\alpha-\beta-1}\;e^{-t-\frac{p}{t}} \;\; \gamma(\beta,zt)\;dt,
\end{align}
and
\begin{align}\label{J-integral}  
&\;_{1}\Gamma _{1} \left[\begin{array}{rr}
(\alpha,x;p);          \\
\gamma+1;    \\
\end{array} -z \right]= \frac{\Gamma (\gamma+1)}{\Gamma (\alpha)} z^{-\frac{\gamma}{2}} \int_{x}^{\infty } t^{\alpha-\frac{\gamma}{2}-1} \;e^{-t-\frac{p}{t}} \;\; J_{\gamma} (2 \sqrt{zt}) dt
\end{align}
provided that the integrals involved are convergent.
\end{corollary}
Since the incomplete gamma function $\gamma(\alpha,x)$ is connected to the error function $\mathrm{erf}(\sqrt{z})$ by means of the relation
(see, \emph{e.g.}, ~\cite{Ol-Lo-Bo}):
\begin{equation}\label{error}
\gamma\left(\frac{1}{2},x\right)=\sqrt{\pi} \;\mathrm{erf}(\sqrt{z}),
\end{equation}
therefore, by applying the relationship \eqref{error}  to \eqref{gamma-integral}
with $\beta=\frac{1}{2}$, we get the following integral representation for the extended incomplete hypergeometric function
\eqref{incomplete-Gauss-2f1}.
\begin{corollary}\label{Gam-2}
  The following integral relationship in terms of the error function $\mathrm{erf}(z)$
holds true:
\begin{align}\label{error-integral}
&\;_{2}\Gamma _{1} \left[\begin{array}{rr}
(\alpha,x;p),\frac{1}{2};          \\
\frac{3}{2};    \\
\end{array} -z \right] = \frac{1}{\Gamma (\alpha)}\sqrt{\frac{\pi}{z}} \int_{x}^{\infty}t^{\alpha-\frac{3}{2}}\;e^{-t-\frac{p}{t}} \;\;  \mathrm{erf}(\sqrt{zt}) \;dt
\end{align}
\end{corollary}
\begin{theorem} The following derivative formulas hold true:
\begin{align}\label{derivative-formula}
&\hspace{-4mm}\frac{d^{n}}{d z^{n}}\left\{\;_{r}\Gamma _{s} \left[\begin{array}{rr}
(\alpha_{1},x;p),\alpha_{2},\cdots,\alpha_{r};          \\
\beta_{1},\cdots,\beta_{s};    \\
\end{array} z \right]\right\} \notag \\
&\qquad= \frac{(\alpha_{1})_{n}\cdots(\alpha_{r})_{n}}{(\beta_{1})_{n}\cdots(\beta_{s})_{n}}\; _{r}\Gamma _{s} \left[\begin{array}{rr}
(\alpha_{1}+n,x;p),\alpha_{2}+n,\cdots,\alpha_{r}+n;          \\
\beta_{1}+n,\cdots,\beta_{s}+n;    \\
\end{array} z \right].
\end{align}
provided that each member of the assertion \eqref{derivative-formula} exists.
\end{theorem}
\begin{proof} Differentiating \eqref{eighf-2} with respect to $z$ and then replacing $n\mapsto n+1$ in the right-hand side term, we get
\begin{align}\label{D-Gauss-pFq-1}
&\frac{d}{dz} \left\{\;_{r}\Gamma _{s} \left[\begin{array}{rr}
(\alpha_{1},x;p),\alpha_{2},\cdots,\alpha_{r};          \\
\beta_{1},\cdots,\beta_{s};    \\
\end{array} z \right]\right\} =\sum_{n=0}^{\infty} \frac{[\alpha_{1};x,p]_{n+1}(\alpha_{2})_{n+1}\cdots(\alpha_{r})_{n+1}}{(\beta_{1})_{n+1}\cdots(\beta_{s})_{n+1}}\; \frac{z^{n}}{n!}\notag \\
&\qquad\qquad  \hskip 20mm= \frac{\alpha_{1}\alpha_{2}\cdots \alpha_{r}}{\beta_{1}\beta_{2}\cdots \beta_{s}}\;_{r}\Gamma _{s} \left[\begin{array}{rr}
(\alpha_{1}+1,x;p),\alpha_{2}+1,\cdots,\alpha_{r}+1;          \\
\beta_{1}+1,\cdots,\beta_{s}+1;    \\
\end{array} z \right],
\end{align}
where we have used the identity \eqref{Pochhammer-relation}. Repeated procedure $n$-times gives the formula \eqref{derivative-formula}.
\end{proof}
\section{ Certain Generating Functions}
In order to derive certain generating functions, we find it to be convenient to choose the abbreviated notation $\Delta (N; \lambda )$ which stands for the array of $N$-parameters:
$$ \frac{\lambda}{N}, \frac{\lambda+1}{N},\cdots, \frac{\lambda+N-1}{N} \qquad (\lambda \in \mathbb{C};\, N \in \mathbb{N}),$$
the array $\Delta(N; \lambda )$ is understood to be empty when $N = 0$. We first establish the following generating function. For $p=x=0$, the corresponding deduced results below are mentioned in \cite{Sr-Ma}.
\begin{theorem}\label{GF-1} Let $p\geqq0;\;\mid t \mid < 1; \lambda \in \mathbb{C}; N \in \mathbb{N} $, then
\begin{align}\label{Generating-function-1}  
&\sum\limits_{n=0}^{\infty} \frac{(\lambda )_{n}}{n!} \, _{r+N}\Gamma _{s}\left[\begin{array}{rr}
\Delta (N; \lambda +n), (\alpha_{1},x;p),\alpha_{2},\cdots,\alpha_{r};  \\
     \beta_{1},\cdots,\beta_{s};  \\
\end{array} z \right] t^{n}\notag\\
&\hskip 30mm =  (1-t)^{-\lambda}\,  _{r+N}\Gamma_{s} \left[\begin{array}{rrr}
\Delta (N; \lambda), (\alpha_{1},x;p),\alpha_{2},\cdots,\alpha_{r};   \\
        \beta_{1},\cdots,\beta_{s};  \\
\end{array} \frac{z}{(1-t)^{N}} \right]
\end{align}
provided that each member of \eqref{Generating-function-1} exists.
\end{theorem}
\begin{proof} Let $\mathcal{S}$ be the left-hand side of \eqref{Generating-function-1}. Using \eqref{eighf-2}, we have
\begin{align}\label{Generating-function-1-a}
&\mathcal{S}=\sum_{n=0}^{\infty}\frac{(\lambda)_n}{n!} \left( \sum_{m=0}^{\infty} \frac{(\lambda+n)_{Nm}[\alpha_{1};x,p]_{m}(\alpha_{2})_{m}\cdots(\alpha_{r})_{m}}{(\beta_{1})_{m}\cdots(\beta_{s})_{m}}\;  \frac{z^{m}}{m!}\right) t^n\notag\\
&\hskip 5mm=\sum_{m=0}^{\infty}\frac{(\lambda)_{Nm}[\alpha_{1};x,p]_{m}(\alpha_{2})_{m}\cdots(\alpha_{r})_{m}}{(\beta_{1})_{m}
\cdots(\beta_{s})_{m}}\;  \frac{z^{m}}{m!}\left(\sum_{n=0}^{\infty} (\lambda+Nm)_{n}\frac{t^n}{n!}\right),
\end{align}
where we have changed the order of summation and used the identity:
$$(\lambda)_{n}(\lambda+n)_{Nm}=(\lambda)_{Nm}(\lambda+Nm)_{n}.$$
Applying now the binomial summation:
$$(1-t)^{-\lambda-Nm}=\sum_{n=0}^{\infty}\frac{(\lambda+Nm)_{n}}{n!}t^n \qquad (|t|<1)$$
in \eqref{Generating-function-1-a} and using again \eqref{eighf-2}, we arrive at the desired result \eqref{Generating-function-1} of Theorem \ref{GF-1}.
\end{proof}
\begin{theorem}\label{GF-2} Each of the following generating functions hold true:
\begin{align}\label{Generating-function-2}  
&\sum\limits_{n=0}^{\infty} \frac{(\lambda )_{n}}{n!} \, _{r+N}\Gamma_{s} \left[\begin{array}{rr}
\Delta (N; -n), (\alpha_{1},x;p),\alpha_{2},\cdots,\alpha_{r};  \\
     \beta_{1},\cdots,\beta_{s};  \\
\end{array} z \right] t^{n}\notag \\
&\hskip 25mm=  (1-t)^{-\lambda} \,  _{r+N}\Gamma_{s} \left[\begin{array}{rr}
\Delta (N; \lambda), (\alpha_{1},x;p),\alpha_{2},\cdots,\alpha_{r};  \\
     \beta_{1},\cdots,\beta_{s};  \\
\end{array} z \left(-\frac{t}{(1-t)} \right)^{N} \right],
\end{align}
\begin{align}\label{Generating-function-3}  
&\sum\limits_{n=0}^{\infty} \frac{(\lambda )_{n}}{n!} \, _{r+2N}\Gamma_{s} \left[\begin{array}{rr}
\Delta (N; -n), \Delta (N; \lambda+n), (\alpha_{1},x;p),\alpha_{2},\cdots,\alpha_{r};  \\
     \beta_{1},\cdots,\beta_{s};  \\
\end{array} z \right] t^{n}\notag \\
&\hskip 25mm= (1-t)^{-\lambda} \,  _{r+2N}\Gamma_{s} \left[\begin{array}{rr}
\Delta (2N; \lambda), (\alpha_{1},x;p),\alpha_{2},\cdots,\alpha_{r};  \\
     \beta_{1},\cdots,\beta_{s};  \\
\end{array} z \left(-\frac{4t}{(1-t)^{2}} \right)^{N} \right]
\end{align}
and
\begin{align}\label{Generating-function-4}  
&\sum\limits_{n=0}^{\infty} \frac{(\lambda )_{n}}{n!} \, _{r+N}\Gamma_{s+N} \left[\begin{array}{rr}
\Delta (N; -n), (\alpha_{1},x;p),\alpha_{2},\cdots,\alpha_{r};  \\
\Delta (N; 1-\lambda-n), \beta_{1},\cdots,\beta_{s};  \\
\end{array} z \right] t^{n}\notag \\
&\hskip 30mm=  (1-t)^{-\lambda} \,  _{r}\Gamma_{s} \left[\begin{array}{rr}
(\alpha_{1},x;p),\alpha_{2},\cdots,\alpha_{r};  \\
     \beta_{1},\cdots,\beta_{s};  \\
\end{array} z t^{N} \right],
\end{align}
provided that $\left(p\geqq0;\;\mid t \mid < 1; \lambda \in \mathbb{C}; N \in \mathbb{N}\right)$ such that each member of the assertions \eqref{Generating-function-2} to \eqref{Generating-function-4} exist.
\end{theorem}
\begin{proof}
The generating functions \eqref{Generating-function-2} to  \eqref{Generating-function-4} can be established by following the method of derivation of the generating function \eqref{Generating-function-1}.
The details can thus be omitted.
\end{proof}
Lastly, we give a simple consequence each of Theorem \ref{GF-1} and Theorem \ref{GF-2}. For $N=1$,  \eqref{Generating-function-1} of Theorem \ref{GF-1} and  \eqref{Generating-function-4} of Theorem \ref{GF-2} give the following generating functions.
\begin{align}\label{Generating-function-5}  
&\sum\limits_{n=0}^{\infty} \frac{(\lambda )_{n}}{n!} \, _{r+1}\Gamma_{s} \left[\begin{array}{rr}
\lambda +n, (\alpha_{1},x;p),\alpha_{2},\cdots,\alpha_{r};  \\
     \beta_{1},\cdots,\beta_{s};  \\
\end{array} z \right] t^{n}\notag \\
&\hskip 30mm=  (1-t)^{-\lambda} \,  _{r+1}\Gamma_{s} \left[\begin{array}{rr}
\lambda, (\alpha_{1},x;p),\alpha_{2},\cdots,\alpha_{r};  \\
     \beta_{1},\cdots,\beta_{s};  \\
\end{array} \frac{z}{1-t} \right]
\end{align}
and
\begin{align}\label{Generating-function-7}  
&\sum\limits_{n=0}^{\infty} \frac{(\lambda )_{n}}{n!} \, _{r+2}\Gamma_{s} \left[\begin{array}{rr}
-n, \lambda + n, (\alpha_{1},x;p),\alpha_{2},\cdots,\alpha_{r};  \\
     \beta_{1},\cdots,\beta_{s};  \\
\end{array} z \right] t^{n}\notag \\
&\hskip 30mm=  (1-t)^{-\lambda} \,  _{r+2}\Gamma_{s} \left[\begin{array}{rr}
\Delta (2;\lambda), (\alpha_{1},x;p),\alpha_{2},\cdots,\alpha_{r};  \\
     \beta_{1},\cdots,\beta_{s};  \\
\end{array} -\frac{4zt}{(1-t)^{2}} \right].
\end{align}
\section{ Fractional Calculus Approach}
In this section, we deduce the formulas for the \emph{Riemann-Liouville} fractional integral $I_{a+}^{\mu}$ and the fractional derivative $D_{a+}^{\mu}$ operators for the $_{r}\Gamma_{s}^{p}(z)$ in \eqref{gips-p-2}(see, \emph{e.g.}, \cite{Sa-Ki-Ma}):
\begin{equation}\label{Riemann-Liouville-integral}  
\left(I_{a+}^{\mu} \varphi \right) (y) = \frac{1}{\Gamma (\mu)} \int_{a}^{x} \frac{\varphi (t)}{(y-t)^{1-\mu}} dt \quad \big(\mu \in \mathbb{C},\, \Re(\mu) > 0\big)
\end{equation}
and
\begin{equation}\label{Riemann-Liouville-derivative}  
\left(D_{a+}^{\mu} \varphi \right) (y) = \left(\frac{d}{dy} \right)^{n} \left(I_{a+}^{n-{\mu}} \varphi \right) (y)\quad \big(\mu \in \mathbb{C},\, \Re(\mu) > 0;\,n=[\Re(\mu)]+1\big),
\end{equation}
where $[y]$ means the greatest integer not exceeding real $y$.
\begin{theorem}
 Let $a \in \mathbb{R}_{+}=[0,\infty);\,\rho,\,\mu,\,\omega \in \mathbb{C}\,\, and \,\,\Re(\rho)>0,\,\Re(\mu)>0,\,\Re(p)>0$. Then, for $y>a$, the following relations hold true:
\begin{align}\label{Riemann-Liouville-1}  
&\left(I_{a+}^{\mu} \left\{(t-a)^{\rho-1}\,\;_{r}\Gamma _{s} \left[\begin{array}{rr}
(\alpha_{1},x;p),\alpha_{2},\cdots,\alpha_{r};          \\
\beta_{1},\cdots,\beta_{s};    \\
\end{array} \omega(t-a) \right] \right\} \right) (y)\notag \\
&\qquad\hskip 25mm =\frac{(y-a)^{\rho+\mu-1}\Gamma(\rho)}{\Gamma(\rho+\mu)}\;_{r+1}\Gamma _{s+1} \left[\begin{array}{rr}
(\alpha_{1},x;p),\alpha_{2},\cdots,\alpha_{r},\rho;          \\
\beta_{1},\cdots,\beta_{s},\rho+\mu;    \\
\end{array} \omega(y-a) \right],
\end{align}
and
\begin{align}\label{Riemann-Liouville-2}  
&\left(D_{a+}^{\mu} \left\{(t-a)^{\rho-1}\;_{r}\Gamma _{s} \left[\begin{array}{rr}
(\alpha_{1},x;p),\alpha_{2},\cdots,\alpha_{r};          \\
\beta_{1},\cdots,\beta_{s};    \\
\end{array} \omega(t-a) \right] \right\} \right) (y)\notag \\
&\qquad\hskip 25mm =\frac{(y-a)^{\rho-\mu-1}\Gamma(\rho)}{\Gamma(\rho-\mu)}\;_{r+1}\Gamma _{s+1} \left[\begin{array}{rr}
(\alpha_{1},x;p),\alpha_{2},\cdots,\alpha_{r},\rho;          \\
\beta_{1},\cdots,\beta_{s},\rho-\mu;    \\
\end{array} \omega(y-a) \right]
\end{align}
\end{theorem}
\begin{proof} Making use of \eqref{Riemann-Liouville-integral} and \eqref{gips-p-2} and applying term-by-term fractional integration by virtue of the formula \cite{Sa-Ki-Ma}:
\begin{equation}\label{fractional-integration-formula}
\left(I_{a+} ^{\alpha} [(t-a)^{\beta-1}]\right) (y)=\frac{\Gamma (\beta)}{\Gamma (\alpha+\beta)}(y-a)^{\alpha+\beta-1}\qquad \big(\alpha,\beta \in \mathbb{C},\, \Re(\alpha) > 0,\, \Re(\beta) > 0\big)
\end{equation}
we get for $y>a$:
\begin{align}\label{Riemann-Liouville-4}  
&\left(I_{a+}^{\mu} \left\{(t-a)^{\rho-1}\;_{r}\Gamma _{s} \left[\begin{array}{rr}
(\alpha_{1},x;p),\alpha_{2},\cdots,\alpha_{r};          \\
\beta_{1},\cdots,\beta_{s};    \\
\end{array} \omega(t-a) \right] \right\}  \right) (y)\notag \\
&\qquad\hskip 25mm =\sum_{m=0}^{\infty}\frac{[\alpha_{1};x,p]_{m}(\alpha_{2})_{m}\cdots(\alpha_{r})_{m}}{(\beta_{1})_{m}
\cdots(\beta_{s})_{m}} \frac{\omega^{m}}{m!}\left(I_{a+}^{\mu} \left\{(t-a)^{\rho+ m-1}
\right\}\right)\notag \\
&\qquad\hskip 25mm  =\frac{(y-a)^{\rho+\mu-1}\Gamma(\rho)}{\Gamma(\rho+\mu)}\;_{r+1}\Gamma _{s+1} \left[\begin{array}{rr}
(\alpha_{1},x;p),\alpha_{2},\cdots,\alpha_{r},\rho;          \\
\beta_{1},\cdots,\beta_{s},\rho+\mu;    \\
\end{array} \omega(y-a) \right].
\end{align}
Next, by using \eqref{Riemann-Liouville-derivative} and \eqref{gips-p-2} and taking into account \eqref{Riemann-Liouville-1}, with $\mu$ replaced by $n-\mu$, we have
\begin{align}\label{Riemann-Liouville-5}  
&\left(D_{a+}^{\mu} \left\{(t-a)^{\rho-1} \;_{r}\Gamma _{s} \left[\begin{array}{rr}
(\alpha_{1},x;p),\alpha_{2},\cdots,\alpha_{r};          \\
\beta_{1},\cdots,\beta_{s};    \\
\end{array} \omega(t-a) \right]\right\} \right) (x))\notag \\
&\hskip 18mm  =\left(\frac{d}{dy} \right)^{n}\left(I_{a+}^{n-\mu} \left\{(t-a)^{\rho-1} \;_{r}\Gamma _{s} \left[\begin{array}{rr}
(\alpha_{1},x;p),\alpha_{2},\cdots,\alpha_{r};          \\
\beta_{1},\cdots,\beta_{s};    \\
\end{array} \omega(t-a) \right]\right\} \right) (y)\notag \\
&\hskip 20mm  =\left(\frac{d}{dy} \right)^{n}\left\{\frac{(y-a)^{\rho+n-\mu-1}\Gamma(\rho)}{\Gamma(\rho+n-\mu)} \;_{r+1}\Gamma _{s+1} \left[\begin{array}{rr}
(\alpha_{1},x;p),\alpha_{2},\cdots,\alpha_{r},\rho;          \\
\beta_{1},\cdots,\beta_{s},\rho+n-\mu;    \\
\end{array} \omega(y-a) \right]\right\}\notag \\
&\hskip 18mm  =\frac{\Gamma(\rho)}{\Gamma(\rho+n-\mu)}\sum_{m=0}^{\infty}\frac{[\alpha_{1};x,p]_{m}(\alpha_{2})_{m}\cdots(\alpha_{r})_{m}(\rho)_{m}}{(\beta_{1})_{m}
\cdots(\beta_{s})_{m}(\rho+n-\mu)_{m}} \frac{\omega^{m}}{m!}\left(\frac{d}{dy} \right)^{n}(y-a)^{\rho+n-\mu+m-1}
\end{align}
Differentiating \eqref{Riemann-Liouville-5} term-by-term and using again \eqref{gips-p-2}, we are led to the desired result \eqref{Riemann-Liouville-2}.
\end{proof}

\end{document}